\documentclass{article}

\usepackage{graphicx,amsthm,amssymb,amsmath,authblk}
\usepackage{hyperref}
\usepackage{xcolor,xspace}
\usepackage{comment}

\usepackage{graphicx}
\usepackage{tikz}
\usetikzlibrary{shapes, shadows, positioning}

\usepackage{tikz}
\usetikzlibrary{decorations.pathreplacing,angles,quotes}
\usepackage{tikz-3dplot}
\usetikzlibrary{shapes,calc,positioning,intersections}
\usetikzlibrary {decorations.pathmorphing, decorations.pathreplacing, decorations.shapes, }

\tdplotsetmaincoords{80}{120}

\newtheorem{theorem}{Theorem}
\newtheorem{lemma}[theorem]{Lemma}
\newtheorem{corollary}[theorem]{Corollary}
\newtheorem{proposition}[theorem]{Proposition}

\newcommand{\whp}{\text{~w.h.p.}\xspace}
\newcommand{\uar}{\text{~u.a.r.}\xspace}

\newcommand{\remove}[1] {}

\newcommand{\cB}{\mathcal{B}}

\newcommand{\cN}{\mathcal{N}}

\renewcommand{\log}{\ln}

\def\E{\mathbb E}

\hypersetup{colorlinks=true,
  citecolor=blue,
   filecolor=blue,
   linkcolor=blue,
   urlcolor=blue
}

\title{Rainbow connectivity of multilayered random geometric graphs
}

\author[1]{Josep D\'iaz\thanks{Email: diaz@cs.upc.edu Research of J.~D. is supported by the Spanish Ministery of Science \& Innovation MCIN/AEI/10.13039/501100011033 (Project PID2020-112581GB-C21 MOTION)}}
\author[2]{Öznur Yaşar Diner\thanks{Email: oznur.yasar@khas.edu.tr Research of Ö.~Y. is supported by by the Spanish Ministery of Science \& Innovation MCIN/AEI/10.13039/501100011033 (Project PID2020-112581GB-C21 MOTION)}}
\author[1]{Maria Serna\thanks{Email: mjserna@cs.upc.edu Research of M.~S. is supported by by the Spanish Ministery of Science \& Innovation MCIN/AEI/10.13039/501100011033 (Project PID2020-112581GB-C21 MOTION)}}
\author[1]{Oriol Serra\thanks{Email: oriol.serra@upc.edu Research of O.~S. is supported by the Spanish Agencia Estatal de Investigaci\'{o}n  [PID2020-113082GB-I00, CONTREWA]}}

\affil[1]{Universitat Polit\`{e}cnica de Catalunya, Barcelona}
\affil[2]{Kadir Has University, Istanbul}

\begin{document}

\maketitle
\begin{abstract}


An edge-colored multigraph $G$ is rainbow connected if every pair of vertices is joined by at least one rainbow path, i.e., a path where no two edges are of the same color. 
 In the context of multilayered networks we introduce the notion of multilayered  random geometric graphs, from  $h\ge 2$  independent random geometric graphs $G(n,r(n))$ on the unit square. We define an edge-coloring by coloring the edges according to the copy of $G(n,r(n))$ they belong to and study the rainbow connectivity of the resulting edge-colored multigraph. We show that $r(n)=\left(\frac{\log n}{n}\right)^{\frac{h-1}{2h}}$ is a threshold of the radius for the property of being rainbow connected.  This complements  the known analogous results for the multilayerd graphs defined on  the Erd\H{o}s-R\'{e}nyi random model.
\end{abstract}

\section{Introduction}\label{sec:intro}
Complex networks are used to simulate large-scale real-world systems, which may consist of various interconnected sub-networks or topologies. For example, this could include different transportation systems and the coordination of their schedules, as well as modeling interactions across different topologies of the network. Barrat et al.  
proposed a new network model to more accurately represent the emerging large network systems, which include coexisting and interacting different topologies \cite{Barrat2004}. Those network models are known as  {\em layered complex networks, multiplex networks} or as {\em multilayered networks}. 

In a multilayered network, each type of interaction of the agents gets its own layer, like a social network having a different layer for each relationship, such as friendship or professional connections 
\cite{dickison2016}.

Recently, there's been a lot of interest in adapting tools used in the analysis for single-layer networks to the study of multilayered  ones, both in deterministic and random models \cite{bianconi2018}.
In the present work, we introduce a random model of colored multigraphs, the \emph{multilayered random geometric graphs} and explore thresholds on their radius for being {\em rainbow connected}.
\subsection{Random Geometric Graphs} 
A {\em random geometric graph} , $G(n,r(n))$ on the 2-dimensional unit square $I^2=[0,1]^2$ is defined as follows: Given a set 
of $n$ vertices $V$ and a radius $r(n)\in[0,\sqrt{2}]$,
the vertices are sprinkled uniformly at random (\uar) in the unit square $I^2$.  Two vertices $u,v$ are adjacent if and only if their Euclidean distance is less than or equal to $r(n)$. In this paper, when it is clear from context, we shorten $r(n)$ to $r$.

Random geometric graphs provide a natural framework for the design and analysis of relay stations and wireless networks.  
Gilbert initially proposed this model to simulate the placement of relays between telephone stations \cite{gilbert}. His work  is also considered as the beginning of continuum percolation theory.

Since the introduction of wireless communication, random geometric graphs have been one of the main models for these communication networks, leading to a fruitful line of research in this field; see, for example~\cite{Duchemin}. 

For further information on random geometric graphs, one may refer to the book by Penrose \cite{Penrose} or to the more recent survey by Walters \cite{Walters}. 

Random geometric graphs exhibit a sharp threshold behavior with respect to monotone increasing graph properties~\cite{goel2005}. In particular, for connectivity,  as the value of $r(n)$ increases,  there is a critical threshold value $r_c(n)$ such that when $r(n) < r_c(n)$, the graph is typically disconnected, while for $r(n) > r_c(n)$, the graph is typically connected. The threshold for connectivity of $G(n,r(n))$ appears at $r_c(n)\sim \sqrt{\frac{\ln n -\ln\ln n}{\pi n}}$, and it is a sharp threshold, in the sense that the threshold's window could be made as small as one wishes. Notice $r_c(n)$ is also a threshold for the disappearance of isolated vertices in $G(n,r(n))$. 
Regarding the diameter of a random geometric graph $G(n,r(n))$,  D\'{\i}az et al. showed that, if $r(n)=\Omega(r_c(n))$, then the diameter of $G(n,r(n))$ is $(1+o(1))\frac{\sqrt{2}}{r(n)}$~\cite{Diaz2016}.

Given a vertex $v$ in a random geometric graph,  $G=G(n,r(n))$, we use $\cB_{G} (v)$
to denote the set of vertices that fell inside the circle  centered at $v$ with radius $r(n)$, the {\em `v-ball'}, and it represents  the set of vertices connected to $v$ in $G(n,r(n))$. So for any given vertex $v$ in $G(n,r(n))$,  $|\cB_{G} (v)|$ is the number of neighbors of $v$ in $G(n,r)$.

Given a random geometric graph $G(n,r(n))$ on $I^2$,  for any vertex  $v\in V(G)$,  $\cB_{G}(v)$ is defined by a binomial random process in $V\setminus \{v\}$, where $u\not= v$
belongs to $\cB_{G}(v)$ with probability $$\frac{\pi r(n)^2}{4}\le \Pr(u\in \cB_{G} (v))\le \pi r(n)^2.$$ The lower bound is due to the {\em `boundary effect'} in $I^2$: If $v$ is at one of the  corners in $I^2$, the part of the area inside the circle  centered at $v$ with radius $r(n)$ and $I^2$ is $\frac{\pi r(n)^2}{4}$. Define the random variable $Z_v=|\cB_{G}(v)|$. Therefore, 
\begin{equation}\label{eq:expetball}
\frac{\pi n r(n)^2}{4}\le \E(Z_v) \le \pi n r(n)^2 \,.
\end{equation}
It is also well-known  that, using Chernoff's bounds, the distribution of $Z_{v}$ is concentrated around its expectation, namely for $0<\epsilon<1$, 
\begin{equation}\label{eq:chernoff}
\Pr\left(|Z_{v} - \E(Z_{v})|> \epsilon \,\E(Z_{v})\right)\le 2e^{\frac{-\epsilon^2\E(Z_{v})}{3}}\,.
\end{equation} 
\subsection{Multilayered Random Geometric graphs}\label{subsection:MRG} 
We introduce a general definition for the random model of edge-colored multigraphs obtained by superposition of a collection of random geometric graphs on the same set $V=\{v_1,\ldots v_n\}$  of vertices. Formally, a \emph{multilayered geometric graph} $G(n, r(n), h, b)$ is defined by four parameters,  $n=|V|$,
$r(n)$ the radius of connectivity of all the random geometric graphs, and $h$ the number of layers,  together with a function that gives  the position assignment $b:[n]\to \underbrace{I^2 \times \dots \times I^2}_h$.
For $v_i\in V$, we denote $b(i)=(b_1^i,\dots, b_h^i)$, where $b_k^i\in [0,1]^2$, for $k\in [h]$ and $i\in[n]$, gives the position of vertex $v_i$ in $G_k$. 

The colored multigraph  $G=G(n,r(n),h,b)$ has  vertex set $V$ and for $k\in [h]$,  there is an edge $(v_i,v_j)$ with color $k$,  if 
the Euclidean distance between $b_k^i$ and $b^j_k$ is at most $r(n)$. 
See Figure~\ref{fig:2layer} for an example  with two layers. 

Note that, for $k\in [h]$, $r(n)$ and  positions $(b_k^i)_{i=1}^n$, a geometric graph $G_k(n,r)$ is defined by the edges with color $k$. Thus, $G(n,r(n),h,b)$ can be seen as the colored union of $h$  geometric graphs, all with the same vertex set and radius. Observe that $G(n,r(n),h,b)$ is defined in $I^{2h}$.  We refer to $G_k(n,r)$ as the $k$-th layer of $G(n,r(n),h,b)$. We denote the $v$-ball $\cB_{G_i}(v)$ of a vertex $v$ in the $i$--th. layer as $\cB_{i}(v)$.

A \emph{multilayered random geometric graph} 
$G(n, r(n),h)$ is obtained  when 
the position assignment $b$ of the vertices is selected independently, for each vertex and layer,  uniformly at random in $[0,1]^2$. 
Thus, the $k$-th layer of $G=G(n,r(n), h)$ is a random geometric graph and $G$ is the colored  union of $h$ independent random geometric graphs.  This definition is given for $I^2$, as in the present paper we only consider geometric graphs on a 2-dimensional space. However, notice that the previous definition  can be extended to geometric graphs in a $\delta$-dimensional space by redefining the scope of the position function to $[0,1]^\delta$. However, in this paper we deal only with the case $\delta = 2$.

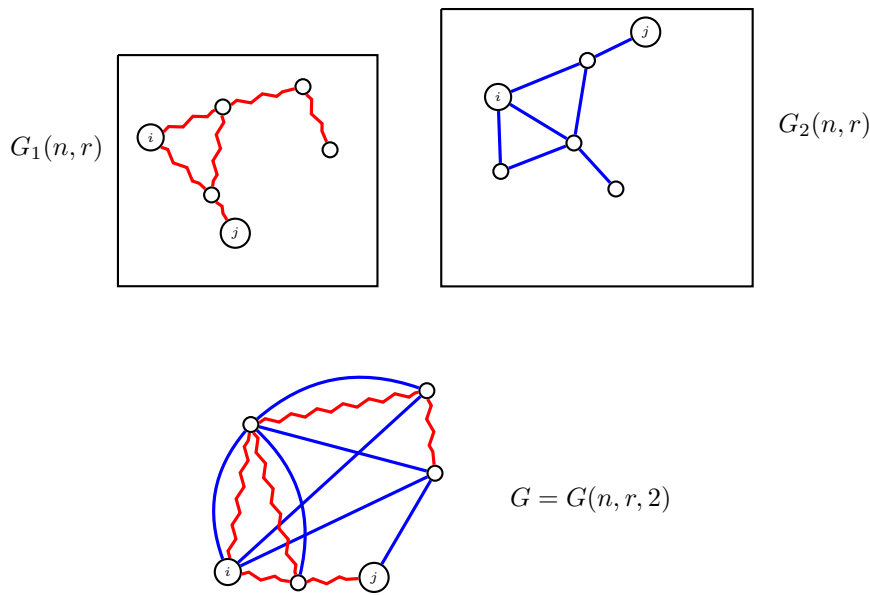
\begin{figure}[t]
\begin{center}
\begin{tikzpicture}[scale=1.5,thick]
		\tikzstyle{every node}=[minimum width=0pt, inner sep=2pt, circle]
			\draw (-4.01,1.27) node[draw] (0) {\tiny $i$};
			\draw (-3.37,1.54) node[draw] (1) {};
			\draw (-2.66,1.72) node[draw] (2) {};
			\draw (-2.42,1.16) node[draw] (3) {};
			\draw (-3.47,0.76) node[draw] (4) {};
			\draw (-3.26,0.42) node[draw] (5) {\tiny $j$};
			\draw (-4.3,2)  -- (-4.3,-0.05) -- (-2,-0.05) --(-2,2)  -- (-4.3,2);
			\draw[red,very thick, decorate, decoration={zigzag,amplitude=.4mm}]   (0) -- (4);
			\draw[red,very thick, decorate, decoration={zigzag,amplitude=.4mm}]   (0) -- (1);
			\draw[red,very thick, decorate, decoration={zigzag,amplitude=.4mm}]    (1) -- (4);
			\draw[red,very thick, decorate, decoration={zigzag,amplitude=.4mm}]   (4) -- (5);
			\draw[red,very thick, decorate, decoration={zigzag,amplitude=.4mm}]   (1) -- (2);
			\draw[red,very thick, decorate, decoration={zigzag,amplitude=.4mm}]   (2) -- (3);
                \draw (-4.84,1.17) node {$G_1(n,r)$};
		\end{tikzpicture} \qquad
  		\begin{tikzpicture}[scale=1.5,thick]
		\tikzstyle{every node}=[minimum width=0pt, inner sep=2pt, circle]
			\draw (-3.88,1.35) node[draw] (0) {\tiny $i$};
			\draw (-3.32,1.01) node[draw] (1) {};
			\draw (-3.86,0.8) node[draw] (2) {};
			\draw (-3.22,1.62) node[draw] (3) {};
			\draw (-3.01,0.67) node[draw] (4) {};
			\draw (-2.79,1.83) node[draw] (5) {\tiny $j$};
                \draw (-4.3,2)  -- (-4.3,-0.05) -- (-2,-0.05) --(-2,2)  -- (-4.3,2);
                \draw[very thick,blue]  (3) edge (5);
			\draw[very thick,blue]  (0) edge (2);
			\draw[very thick,blue]  (0) edge (1);
			\draw[very thick,blue]  (0) edge (3);
			\draw[very thick,blue]  (1) edge (4);
			\draw[very thick,blue]  (1) edge (2);
			\draw[very thick,blue]  (1) edge (3);
                \draw (-1.4643,1.14) node {$G_2(n,r)$};
		\end{tikzpicture}\vskip 1cm 
  		\begin{tikzpicture}[scale=1.5,thick]
		\tikzstyle{every node}=[minimum width=0pt, inner sep=2pt, circle]
			\draw (-3.91,0.6) node[draw] (0) {\tiny $i$};
			\draw (-3.74,1.69) node[draw] (1) {};
			\draw (-2.44,1.94) node[draw] (2) {};
			\draw (-2.38,1.33) node[draw] (3) {};
			\draw (-3.39,0.52) node[draw] (4) {};
			\draw (-2.83,0.56) node[draw] (5) {\tiny $j$};
			\draw[very thick,blue]  (3) edge (5);
			\draw[very thick,blue]  (0) edge (2);
			\draw[very thick,blue] (0) edge [bend left] (1);
			\draw[very thick,blue]  (0) edge (3);
			\draw[very thick,blue]  (1) edge [bend left] (4);
			\draw[very thick,blue]  (1) edge [bend left] (2);
			\draw[very thick,blue]  (1) edge (3);
                \draw[red,very thick, decorate, decoration={zigzag,amplitude=.4mm}]   (0) -- (4);
			\draw[red,very thick, decorate, decoration={zigzag,amplitude=.4mm}]   (0) -- (1);
			\draw[red,very thick, decorate, decoration={zigzag,amplitude=.4mm}]   (1) -- (4);
			\draw[red,very thick, decorate, decoration={zigzag,amplitude=.4mm}]   (4) -- (5);
			\draw[red,very thick, decorate, decoration={zigzag,amplitude=.4mm}]   (1) -- (2);
			\draw[red,very thick, decorate, decoration={zigzag,amplitude=.4mm}]   (2) -- (3);
			\draw (-1.23,1.14) node  {$G=G(n,r,2)$};
		\end{tikzpicture}

  \end{center}
 \caption{A multilayered random geometric graph $G=G(n,r,2)$ on $I=[0, 1]$ (at the bottom) and its layers which are monochromatic random geometric graphs (at the top.) \label{fig:2layer}}
\end{figure}
\subsection{Rainbow Connectivity}

Given an edge-colored multigraph $G$, we say that $G$ is  {\em rainbow connected} if between any pair of vertices $u,v\in V(G)$, there is a path, called  a \emph{rainbow path}, with edges of pairwise distinct colors. Observe that the colored multigraph $G$ given in Figure~\ref{fig:2layer} is not rainbow connected as there is no rainbow path from the vertex $i$ to the vertex $j$.
 
The \emph{rainbow connection number} of a connected graph $G$ is the minimum number of colors for which $G$
admits a (not necessarily proper) 
edge-coloring such that $G$ is rainbow connected. 
Chartrand et al.~\cite{Chartraud2009} introduced the study of the rainbow connectivity of graphs as a  property to secure strong connectivity in graphs and networks. Since then, variants of rainbow connectivity have been applied to different
deterministic as well as random models of graphs. See the survey of Li et al.~\cite{Li2013}, for further details on the extension of rainbow connectivity to other graph models.

The study of rainbow connectivity has been addressed in the context of {\em multilayered binomial random graphs} by Bradshaw and Mohar~\cite{bradshaw2021}, where the authors give sharp concentration results for three values of the number $h$ of layers needed to ensure rainbow connectivity of the resulting multilayered binomial random graph $G(n,p)$, for the appropriate values of $p$. Those results have been extended by Shang~\cite{Shang2023} to ensure $k$-rainbow connectivity in the same model, namely, the existence of $k$ internally disjoint rainbow paths joining every pair of vertices in the multilayered graph.

In the present paper, we are interested in studying the rainbow connectivity of a multilayered random geometric graph $G(n,r(n),h)$. 
In particular, for every fixed $h$, we are interested in the minimum value of $r(n)$  such that \whp\footnote{As usual \whp means {\em with high probability}, i.e., with probability tending to $1$ as $n\to\infty$.} the multilayered random geometric graph $G(n,r(n),h)$ is rainbow connected. In the same way, for fixed values of $r$, we want to determine the minimum number of layers $h$ such that $G(n,r(n),h)$ is \whp rainbow connected. The second parameter can be defined as the rainbow connectivity of the multilayered random geometric graph.
\section{Main results} 
Our main results are the lower and upper bounds of the value of $r(n)$, to asymptotically ensure  that \whp , $G(n,r(n),h)$ has  or does not have the property of being rainbow connected. 
Through the paper, we assume that  $r(n) \to 0$ as $n\to \infty$.

\begin{theorem}\label{thm:main1} Let $h\ge 2$  be an integer and let $G=G(n,r(n),h)$ be an $h$-layered random geometric graph. Then,  if
$$
r(n)\ge b\left(\frac{\log n}{n^{h-1}}\right)^{1/2h},
$$
for $b>\left(\frac{2^{2+3(h-1)}}{\pi^3}\right)^{1/(2h)}$,
then \whp $G(n,r(n),h)$ is rainbow connected. 
\vspace{0.5cm}

\noindent
Moreover, if
$$
r(n)\le c\left(\frac{\log n}{n^{h-1}}\right)^{1/2h},
$$
for $c<\frac{2}{3\pi}\left(\frac{1}{2\pi}\right)^{h-1}$,
then \whp $G$ is not rainbow connected.
\end{theorem}
For a mutilayered random geometric graph, the
property of 
being rainbow connected is  monotone increasing  on $r$. We will  implicitly use this fact in the proofs  of the threshold of rainbow connectivity stated in Theorem~\ref{thm:main1}.

Notice that Theorem~\ref{thm:main1} can be restated as a threshold on $h$ for the existence \whp of the rainbow connectivity property of a multilayered geometric random graph $G(n,r,h)$, for  a given radius $r(n)$.

\begin{corollary}\label{cor:h} Let $r(n)=o(1)$. Set 
\begin{align*}
h_0&=\left\lfloor\frac{\log n+\log\log n}{2 \ln(r(n))+ \ln n - \ln 4 + \ln (3\pi)}\right\rfloor,\\
h_1&=\left\lceil\frac{\log n+\log\log n - \ln (2\pi^3)}{2 \ln(r(n))- \ln n - \ln 8}\right\rceil.
\end{align*}
Then, the multilayered random geometric graph $G(n,r(n),h)$ is \whp rainbow connected if $h\ge h_1$,  while if $h<h_0$, $G(n,r(n),h)$ is \whp    not rainbow connected.
\end{corollary}

In the next sections, we prove Theorem~\ref{thm:main1}. 
The intuition behind the proof of our result is to find the adequate graph property on which to base the threshold. One can expect multilayered random geometric graphs to have better expansion properties than  random geometric graphs.  Since  the vertices are random scattered in each layer, they have better chances to connect to other vertices than in a random geometric  graph.  This happens even when only rainbow paths are considered to connect vertices.  The natural expansion factor to look at is $nr^2$, which is of the same order of magnitude as the expected size of a neighborhood in a random geometric graph.  The radius $r(n)$ in our threshold, confirms the property that the number of  neighbors in the next layer grows with a factor $nr^2$ until the one before the last.  We use this property to show our main result.

The proof of  case $h=2$ requires a special argument, given in Section \ref{sec:h=2}, that allow us  pinpoint better values for the constants in the threshold than the ones described in Theorem~\ref{thm:main1} for the general case. In Section~\ref{sec:exp}, which is of independent interest, we analyze the local expansion. 
Section \ref{sec:h>2} contains the proof of the lower bound in Theorem \ref{thm:main1}  for the case $h\ge 3$. The proof of the upper bound is given in Section \ref{sec:upper}. The paper concludes with some final remarks and open problems.

\section{Rainbow Connectivity of Two-layered Random Geometric Graphs}\label{sec:h=2}

The proof of Theorem \ref{thm:main1} for the case $h=2$  is given in Proposition \ref{prop:main1h=2} below. Next, we recall the upper and lower bounds for the probability that two vertices $v_i,v_j$ are adjacent in a random geometric graph $G(n,r)$. The lower bound takes into account the boundary effect. 
The next result is a refinement of Equation~\eqref{eq:expetball}.
\begin{lemma} Let $v_i, v_j$ be two vertices in a random geometric graph $G(n,r(n))$. Let $\cB (v_j)$ be the ball of $v_j$-ball in $G(n,r)$. We have
\begin{equation} \label{eq:basicprob}
\pi r(n)^2-o(r(n)^2)\le\Pr (v_i\in \cB (v_j))\le\pi r(n)^2\, ,
\end{equation}
\end{lemma}
\begin{proof}  
 To get the result, we have to consider the effect on the size of $\cB (v_j))$ whne $v_j$ is closed to the boundary of $I^2$.   Let $R$ be the   subsquare of $I=[0,1]^2$, with sides of length  $1-2r$ and centered at $(1/2,1/2)$. Then, conditioning on the events that $v_j$ falls or not inside of $R$, we have 
\begin{align*}
\Pr (v_i\in \cB (v_j)) &= \Pr (v_i\in \cB (v_j)|v_j\in R) Pr(v_j\in R)\\ &~~~~~+  \Pr (v_i\in \cB (v_j)| v_j\not\in R) Pr(v_j\not\in R)\\
&\geq  \pi r^2 (1-2r)^2 + \frac{\pi r^2}{4}  4r(1-r) \\
& =  \pi r^2 (1-3r + 3r^2) = \pi r^2 - o(r^2).
\end{align*}
Besides, 
\begin{align*}
\Pr (v_i\in \cB (v_j)) &= \Pr (v_i\in \cB (v_j)|v_j\in R) Pr(v_j\in R)\\ &~~~~~+  \Pr (v_i\in \cB (v_j)| v_j\not\in R) Pr(v_j\not\in R)\\
&\leq \Pr (v_i\in \cB (v_j)|v_j\in R) Pr(v_j\in R) = \pi r^2 (1-2r)^2  \leq  \pi r^2.
\end{align*}
\end{proof}
%
\begin{proposition}\label{prop:main1h=2} 
Let $G(n, r(n), 2)$ be a two-layered random geometric graph, with layers $G_1(n,r(n))$ and
$G_2(n,r(n))$.\\
{\bf Lower bound:} If  
$$
r(n)\ge 0.68\left(\frac{\log n}{n}\right)^{1/4},
$$
then \whp $G(n, r(n), 2)$ is rainbow connected.\\
{\bf Upper bound:} If
$$
r(n)\le 0.56 \left(\frac{\log n}{n}\right)^{1/4},
$$
then \whp $G(n, r(n), 2)$ is not rainbow connected.
\end{proposition}


\begin{proof} 
Recall that for multilayered geometric graphs,
the property of being rainbow connected  is monotone, on $r$. Therefore, it is enough to prove the statement for the two extreme values of the radius.

Proof of the lower bound: 
Let $r= b\left(\frac{\log n}{n}\right)^{1/4}$, 
where $b$ is a constant to be determined latter.

For each pair  $v_i,v_j\in V$, let  $X_{v_i,v_j}$ be the indicator random variable
\[
X_{v_i,v_j}=
\begin{cases}
1 & \text{if there is no  rainbow path  between $v_i$ and $v_j$ in $G$}\,,\\
0 & \text{otherwise}\,.
\end{cases}
\]
For $v_k\in V\setminus\{v_i,v_j\}$, let $A_{v_k}$  be the event that  $v_k$  is joined to $v_i$ in $G_1(n,r)$ and to $v_j$ in $G_2(n,r)$  
or vice-versa, namely, 
$$A_{v_k}=\{\{v_i\in\cB_1(v_k)\} \cap \{v_j\in\cB_2(v_k)\}\}\cup
\{\{v_j\in\cB_1(v_k)\} \cap \{v_i\in\cB_2(v_k)\}\}\,.
$$
Using Equation~\eqref{eq:basicprob} we have 
\begin{equation}\label{eq:k}
(\pi r^2-o(r^2))^2\le \Pr (A_{v_k})\le 2(\pi r^2)^2\,.
\end{equation}
Let $A_{v_i,v_j}$ denote the event that $v_i$ and 
$v_j$ are joined by an edge in either $G_1(n,r)$ or in $G_2(n,r)$, that is
\begin{equation*}
A_{v_i,v_j}=\{v_i\in \cB_1(v_j)\} \cup \{v_i\in \cB_2(v_j)\}\,,
\end{equation*}
so that 
\begin{equation}\label{eq:ij}
\pi r^2-o(r^2)\le \Pr (A_{v_i,v_j})\le 2\pi r^2\,.
\end{equation}

For given $v_i$ and $v_j$, the event that  they are joined by a rainbow path in $G$ is $(\cup_{k\neq i,j}A_{v_k})\cup A_{v_i, v_j}$. Therefore, since for $v_k\in V\setminus \{v_i,v_j\}$ the events $A_{v_k}$  and $A_{v_i,v_j}$ are independent, using the lower bounds in equations~(\ref{eq:k}) and (\ref{eq:ij}), we have 
\begin{align*}
\E(X_{v_i,v_j}) 
& = \Pr(\overline{(\cup_{v_k\neq v_i,v_j}A_{v_k})\cup (A_{v_i,v_j})}) \\
&= \Pr ((\cap_{v_k\neq v_i,v_j}\overline{A_{v_k}})\cap \overline{(A_{v_i,v_j}})) \\
& \le  (1-(\pi r^2-o(r^2))^2)^{n-2} (1- (\pi r^2-o(r^2))) \\
& \le (1-\pi^2 r^4+o(r^4))^n\,. 
\end{align*}

Recall that $r= b\left(\frac{ln (n)}{n}\right)^{1/4}$ and  
 define $X=\sum_{i<j} X_{v_i,v_j}$ as the random variable that counts the number of pairs 
 $\{v_i,v_j\}$ in $G$, which are not joined by a rainbow path.  Plugging  the value of $r$ in the  expression for $\E(X_{v_i,v_j})$, we obtain 
\begin{align*}
\E(X) =\sum_{i<j}\E (X_{v_i,v_j})&\le {n\choose 2}\left(1-(\pi^2 r^4)+o(r^4)\right)^n \\
& \le e^{2\log n}\left(1-\pi^2b^4\frac{\log n}{n}+o\left(\frac{\log n}{n}\right)\right)^n\\
& \le  e^{(2-b^4\pi^2)\log n+o(\log n)}\,.
\end{align*}
To obtain the value of the constant $b$, we need $\E(X)\to 0$ as $n\to \infty$,  therefore
$2- \pi^2b^4>0$, which implies that $b=0.68$ suffices.

 By Markov's inequality, we get  
 \begin{equation*}
\Pr(X\ge 1)\le \E(X)\to 0, \text{~as~} n\to\infty. 
 \end{equation*}
 Therefore,  \whp the graph $G(n,r(n),h)$ is  rainbow connected for $$r(n)\ge 0.68 \left(\frac{ln (n)}{n}\right)^{1/4}.$$
\vspace*{0.3cm}

To prove  the upper bound for $r$ so that $G(n,r,h)$ is not rainbow connected, consider $r = c(\log n/n)^{1/4}$, where  the value of   constant $0<c<0.68$ will be specified later.
Given a vertex $v_i$,   define a random variable $Y_{v_i}$ that counts  the vertices  $v_j\neq v_i$, such that there is no rainbow path between $v_j$ and $v_i$. We wish to prove   that for the given values of $r$,  $\E(Y_{v_i})\to \infty$ 
as $n\to \infty$.

Given $v_i\in V$ recall the   random variable $Z_{v_i}= |\cB_1(v_i)|$. Let $\mu_i= \E(Z_{v_i})$. From equations~(\ref{eq:expetball}) and (\ref{eq:chernoff}) we get  that 
$\frac{n\pi r^2}{4}\le \ \mu_i\le n\pi r^2$ and that  the distribution of $Z_{v_i}$ is concentrated around $\mu_i$. 

Given vertices $v_i$ and $v_j$ with $v_i\not= v_j$, define the indicator random variable that there is no rainbow path between $v_i$ and $v_j$:
\[
Y_{v_i,v_j}=
\begin{cases}
1 & \text{if~}  \cB_1 (v_i) \cap \cB_2 (v_i)=\emptyset \,,\\
0 & \text{otherwise}\, .
\end{cases}
\]
%
Computing probabilities:
\begin{equation*}\label{eq:yij}
\begin{split}
\Pr(Y_{v_i,v_j}=1) &= \sum_{k=1}^n \Pr(Y_{v_i,v_j}=1\,\mid Z_{v_i}=k)\cdot \Pr(Z_{v_i}=k) \\ 
&= \sum_{k=1}^n \left(1- \pi r^2\right)^{k+1}\cdot  \Pr(Z_{v_i}=k) \\
&\ge \sum_{k=\lceil\frac{1}{2}\mu_i\rceil}^{\lfloor\frac{3}{2}\mu_i\rfloor} (1-\pi r^2)^{k+1}\cdot \Pr(Z_{v_i}=k)\, .
\end{split}
\end{equation*}

Notice that as $(1-\pi r^2) < 1$, asymptotically  the terms of the sequence  
$\{(\pi r^2)^k\}_{k=1}^n$  decrease as $k$ increases. Therefore, taking $\epsilon=1/2$ in  Equation~(\ref{eq:chernoff}) 
to bound $\Pr(|Z_{v_i}-\mu_i|> \mu_i/2)$, we have
\begin{equation}\label{eq:Eyij}
\begin{split}
\Pr(Y_{v_i,v_j}=1) &\ge \sum_{k=\lceil\frac{1}{2}\mu_i\rceil}^{\lfloor\frac{3}{2}\mu_i\rfloor}
(1-\pi r^2)^{k+1}\cdot
\Pr(Z_{v_i}=k) \\ %
&\ge (1-\pi r^2)^{\lceil\frac{1}{2}\mu_i\rceil}\cdot
\sum_{k=\lceil\frac{1}{2}\mu_i\rceil}^{\lfloor\frac{3}{2}\mu_i\rfloor}\Pr(Z_{v_i}=k) \\ %
&\ge (1-\pi r^2)^{\lceil\frac{1}{2}\mu_i\rceil}\cdot \left(1-\Pr\left(|Z_{v_i}-\mu_i|)>\frac{1}{2}\mu_i\right)\right)\\%
&\ge ((1-\pi r^2)^{\lceil\frac{1}{2}\mu_i\rceil}\left(1-2e^{\frac{\mu_i}{12}}\right) \\
&\ge ((1-\pi r^2)^{\frac{1}{2}n\pi r^2}\cdot \left(1-2e^{\frac{\mu_i}{12}}\right)\, .
\end{split}
\end{equation}

Given a vertex $v_i$, we define the random variable $Y_{v_i}$ that counts the number of vertices $v_j\in V\setminus \{v_i\}$, for which there is no rainbow path between $v_j$ and $v_i$, 

Notice that given $v_i$ in $G_1(n, r(n))$, for all  $v_j\in V\setminus {v_i}$, then
 $Y_{v_i}=\sum_{v_j}Y_{v_i,v_j}$. 
 
 So that from Equation~(\ref{eq:Eyij}) we get
\begin{equation}\label{eq:EY_i}
\E(Y_{v_j})= 
\sum_{v_j\not\in V\setminus\{v_i\}}\E(Y_{v_i,v_j})
\ge (n-1)(1-\pi r^2)^{\frac{1}{2}n\pi r^2}
\cdot \left(1-2e^{\frac{\mu_i}{12}}\right).
\end{equation}

Recall that for $0\le x\le 0.8$ we have 
$1-x\ge e^{-2x}$. Using this inequality together with the bounds in Equation (\ref{eq:expetball}), in Equation~(\ref{eq:EY_i}), we get
\begin{equation*}\label{eq:final}
\begin{split}
\E(Y_{v_j})&\ge e^{\ln(n-1)}\cdot 
e^{(-2\pi r^2)\cdot(\frac{1}{2}n\pi r^2)}
\cdot \left(1-2e^{\frac{\mu_i}{12}}\right)\\
&=e^{\underbrace{\ln(n-1)(-\pi r^2)(n\pi r^2)}_{(i)}}
\cdot \underbrace{\left(1-2e^{\frac{\mu_i}{12}}\right)}_{(ii)}
\end{split}
\end{equation*}

Substituting $r= c(\frac{\ln n}{n})^{1/4}$ in  $(i)$ and  $(ii)$, we get:
\begin{equation*}
(i) = \ln(n-1)-\pi^2 c^4\ln n\, 
\end{equation*}
which, provided that $(1-\pi^2 c^4)>0$ goes to infinity as $n\to \infty$. 
Notice the condition $(1-\pi^2 c^4)>0$ implies that $c<0.56$.

Regarding the expression $(ii)$ in Equation~\eqref{eq:final}, we have 
\begin{equation*}
\left(1-2e^{\frac{\mu_i}{12}}\right)\ge 
\left(1-2e^{-0.00257\frac{\ln n}{n}}\right),
\end{equation*}
which tends to 1 as $n\to \infty$.

Therefore, we conclude that if  $c\le 0.56$, then $\E(Y_{v_i})\to \infty$ as $n\to \infty$, which concludes the second part of the proof.
    \end{proof}
%
\section{Local Expansion}\label{sec:exp}

A key property of multilayered random geometric graphs is their local expanding properties in the sense of  Proposition~\ref{prop:exp} below. In this section, we use the following concentration result for the size of the image of a random map. For completeness, we include a proof of this statement. 
%
\begin{lemma}\label{prop:sizerm} Let $m$ and $k$ be positive integers, with $m>k$, and let $g:[m]\to [k]$ be a  map in which for $i\in [m]$, $g(i)\in [k]$ is chosen uniformly and independently  at random (\uar). Let  $Y=|g([m])|$ be a random variable estimating the size of the image of $g$. For every $a>0$, we have that
\begin{equation}\label{eq:McD}
\Pr (|Y-\E(Y)|\ge a)\le 2e^{-2a^2/m}.
\end{equation}
Moreover, 
\begin{equation}\label{eq:McD2}
\Pr (|Y-m|> a)\le 2e^{-2\frac{(a-m^{2}/2k)^2}{m}}
\end{equation}
\end{lemma}
\begin{proof}  Let $f:[k]^m\to [k]$ 
be the function that, for an $m$-tuple $(x_1,\ldots ,x_m)$, 
assigns the cardinality of the set $\{x_1,\ldots ,x_m\}$, i.e.,  $|\cup_{i=1}^m \{x_i\}|$ . The function $f$ satisfies the {\em bounded differences condition}: for any $t\in [m]$,  
$$
\max_{x,x'\in [k]}|f(x_1,\ldots ,x_{t-1},x,x_{t+1},\ldots ,x_m)-f(x_1,\ldots ,x_{t-1},x',x_{t+1},\ldots ,x_m)|\le 1.
$$

Let $g:[m]\to [k]$  by and independent random map. Define the random variable $X_i= g(i)$, for $i\in [m]$, and let $Y=f(X_1,\ldots ,X_m)$ be the random variable counting the number of elements in $\text{Im}(g)$. Then,


\begin{equation*}
\E(Y) = \E(f(X_1,\ldots, X_m)).
\end{equation*}

Hence, as the sum of the differences of $f$ is at most  $m$, using 
McDiarmid's inequality~\cite{McDiarmid1989}, we have that $\forall a>0$
\begin{equation*}
\E(|Y-\E(Y)|\ge a)\le 2e^{-2a^2/m},
\end{equation*}
which proves  Equation~\eqref{eq:McD}. 

To prove Equation~\eqref{eq:McD2}, for each $j\in [k]$ consider the indicator random variable 
\[
Y_j=
\begin{cases}
1, & \text{if}\;  j\in \text{Im}(g)\,,\\
0, & \text{otherwise}\,.
\end{cases}
\]
 Then $Y=Y_1+\cdots +Y_{k}$ and  
$$
m\ge \E (Y)=k\left(1-\left(1-\frac{1}{k}\right)^m\right)\ge 
k(1-e^{-\frac{m}{k}})\ge \frac{km}{k}-\frac{m^2}{2k^2}=m-\frac{m^2}{2 k}.
$$
It follows that, using $Y\le m$ together with the triangle inequality, we have the following
$$
\Pr (|Y-m|\ge a)\le \Pr \left(|Y-\E(Y)|\ge a-\frac{m^2}{2k}\right)\le 2e^{\frac{-2(a-\frac{m^2}{2k})^2}{m}}\, .
$$
\end{proof}

Let $G=G(n,r(n),h)$ be a multilayered random geometric graph.  For a permutation $\sigma$ in the set of colors $[h]$, $u\in V$. A \emph{$\sigma$-rainbow path}  is a path in which the colors in the path follow the sequence $\sigma(1), \sigma(2),\dots$. 

For $1\leq \ell\leq h$, let  
\begin{equation}\label{eq:Nell}
\begin{split}
\cN_{\ell,\sigma}(u)=\{v\in G\, |\, & \exists~ \,
\text{a $\sigma$-rainbow path of length $\ell$ from $u$ but} \\
&\nexists~ \text{a $\sigma$-rainbow path of length $\ell-1$ from $u$}\}\, .
\end{split}
\end{equation}
In   the case where   $\sigma$ is the identity, we simply write $\cN_{\ell}(u)$.
\begin{proposition}\label{prop:exp} Let $h\ge 3$ be fixed, let $\sigma$ be a permutation of $[h]$. Let $G=G(n,r(n),h)$ be a multilayered random geometric graph with 
$$
r(n) = \Theta\left(\frac{\log n}{n^{h-1}}\right)^{\frac{1}{2h}},
$$
and let $u\in V(G)$.

Then, for $1\le \ell\le h-1$,  we have \whp 
$$
\frac{1}{4^\ell} \frac{\pi}{2}(nr^2)^\ell \leq |\cN_{\ell,\sigma}(u)| \leq  \left(\frac{3\pi}{2}\right)^\ell (nr^2)^\ell. 
$$
\end{proposition}

\begin{proof}
For $n$ large enough, we can assume that $r(n) = d\left(\frac{\log n}{n^{h-1}}\right)^{\frac{1}{2h}},$ for some constant $d>0$. In the following, without loss of generality, we prove the result when $\sigma$ is the identity.
First note that our radius selection guarantees that 
for $ 1\leq t\leq h-1$,
$(nr^2)^t= d^t (n\ln n) t/h = o((nr^2)^{t})$.

To prove the result, we use the {\em principle of deferred decision}~\cite{Knuth97}, and  define two random processes that provide \whp upper and lower bounds  for the growth rate of  $\cN_\ell(u)$ with respect to  $\cN_{\ell-1}(u)$, for $1\leq \ell<h$.


Define  $\cN_0(u)=\{u\}$. To construct $\cN_1(u)$, throw \uar the  vertices  in $V$ on $I^2$, the unit square for $G_1(n,r)$.  The vertices falling at Euclidean distance $\le r$ from $u$ form the set $\cN_1(u)=\cB_1(u)$.
From  equations \eqref{eq:expetball} and \eqref{eq:chernoff} taking $\epsilon = 1/2$, we have that  \whp 
$$\frac{1}{8} \pi nr^2 \leq |\cN_1(u)|\leq \frac{3}{2} \pi n r^2.$$

Once  $\cN_{\ell-1}(u)$ has been constructed, 
we obtain $\cN_\ell(u)$ by sprinkling \uar the vertices in $\cN_{\ell-1}(u)$ on the $I^2$ which 
defines  $G_{\ell}(n,r)$. Next, we will scatter the vertices in $V\setminus \left(\cup_{t=0}^{\ell-1} \cN_t (u)\right)$  on $I^2$.  We construct  $\cN_{\ell}(u)$ by the  vertices not in $\cN_{\ell-1}(u)$ that are at Euclidean distance at most $r$ from a vertex in $\cN_{\ell-1}(u)$, as they are neighbors in $G_\ell(n,r)$, and we have a rainbow path of length $\ell-1$ from $u$ with the correct color ordering.  
In this way, at the end of the process, we have created $\cN_{h}(u)$. 

\smallskip
To obtain an upper bound to the previous process,  we create from $\cN_{\ell-1}(u)$  a new set of vertices by first marking in red an area of size $|\cN_{\ell-1}(u)| \pi r^2$ in $I^2$ containing the area covered by the disks centered at $v\in \cN_{\ell-1}(u)$. This is equivalent to assume  that the area at distance at most $r$  of any vertex in  $\cN_{\ell-1}(u)$ is the complete disk of radius $r$ centered at the vertex and that all those disks do not intersect. Then, we scatter $n$ points \uar on $I^2$, the vertices not in  $\cN_{\ell-1}(u)$ and some additional points.  

Consider the set $U_{\ell}(u)$ formed by the points that fall within the red area.   In this process, define the random variable $|U_{\ell}(u)|$ counting the expected number of points in the red area. This random variable  follows a binomial distribution, which  is concentrated around its mean value $n|\cN_{\ell-1}(u)| \pi r^2$. Therefore, and using Chernoff's bound with $\epsilon = 1/2$, \whp 
$$|U_{\ell}(u)|\leq \frac{3}{2} \pi |\cN_{\ell-1}(u)|  n r^2.$$  
Recall that the total area in $I^2$ of the disk  around one vertex is upper bounded by $\pi r^2$, then by coupling with the previous process we get $|\cN_\ell(u)| \leq |U_{\ell}(u)|$. 

\smallskip
The second random process provides a lower bound on the increase of $|\cN_{\ell}(u)|$ with respect to the sizes of $\cN_{1}(u), \cdots, \cN_{\ell-1}(u)$, conditioned  on the fact that  \whp  $|\cN_{t}(u)|=\Theta((nr^2)^t)$, for $1\leq t\leq \ell-1$. 

We first subdivide the unit square  into subsquares, each with diagonal of length $r$. Notice  that a disk of radius $r$, centered at any point in a subsquare, covers entirely that subsquare. The number $k$ of subsquares is  $k= \frac{2}{r^2}$, and each subsquare has an area of $\frac{r^2}{2}$. 

We drop \uar the vertices in $\cN_{\ell-1}(u)$  to the unit square, and declare red any subsquare containing  one or more vertices. 

Let $R$ be the random variable counting the number of red subsquares.

Choose a constant $\gamma$ such that  
 $$
\max\{2\ell-1-h,(\ell-1)/2\}<\gamma< \ell-1.
$$ 
Using Proposition~\ref{prop:sizerm} with $m=|\cN_{\ell-1}(u)|$,  $k= \frac{2}{r^2}$, $a=n^{\gamma/h}$ and $\gamma<\ell-1$, 
\begin{equation*}\label{eq:pj-1}
\Pr(R \geq  |\cN_{\ell-1}(u)|- n^{\gamma/h} )\ge 1-2\exp \left(-2\frac{(n^{\gamma/h}-|\cN_{\ell-1}(u)|^2/k)^2}{|\cN_{\ell-1}(u)|}\right).
\end{equation*}
As  $n^{\gamma/h}=o((nr^2)^{\ell-1})$, then \whp  the number of red subsquares is greater or equal than 
$|\cN_{\ell-1}(u)| - o((nr^2)^{\ell-1})$.
\vspace{0.3cm}

Conditioning on the previous fact, let  $A=V\setminus \left(\cup_{t=0}^{\ell-1} \cN_t (u)\right)$. Recall that in the second process, we scattered \uar the points in $A$ into $I^2$. Let $L_{\ell}(u)$  be the set of points that fall into a red subsquare. Then, $|L_\ell(u)|$   follows a binomial distribution with a mean value 
$$ \mu = |A|\left(|\cN_{\ell-1}(u)|- n^{\gamma/h}\right) \frac{r^2}{2}.$$
Using Chernoff's bound, with $\epsilon =1/2$, we get that 
$$\Pr(|L_\ell(u)| < \frac{1}{2} \mu)\to 0.$$ 

By our conditioning on the number of red square  and taking into account that 
$(nr^2)^{\ell-1} = o(n)$,  the number of vertices in $A$ is \whp 
$$|A|= n - \sum_{t=0}^{\ell-1} |\cN_{t}(u)| = n - o((nr^2)^{\ell-1})\leq n.$$ 

Furthermore,  $|\cN_{\ell-1}(u)|- n^{\gamma/h}\leq |\cN_{\ell-1}(u)|$. 

We conclude that 
$$\Pr\left(|L_\ell(u)| <\frac{1}{2}\, |\cN_{\ell-1}(u)| \, \frac{n r^2}{2}\right) \leq \Pr(|L_\ell(u)| < \frac{1}{2} \mu), $$ 
so \whp 
$$|L_\ell(u)| \geq  \frac{1}{2}|\cN_\ell(u)| \frac{n r^2}{2}.$$   

Notice that the area of a red subsquare is smaller than the area of the parts of the disk of radius $r$ that surround  a vertex inside red  subsquare.  By coupling this process with our initial process, we have  $|L_\ell(u)|\leq \frac{1}{2}|\cN_\ell(u)| \frac{r^2}{2} $. 

Putting all together we have that \whp
$$ \frac{1}{4}|\cN_{\ell-1}(u)| nr^2\leq  |\cN_\ell(u)| \leq  \frac{3}{2} \pi |\cN_{\ell-1}(u)|  n r^2, $$
completing the proof.
\end{proof}

\section{Lower bound for $h\ge 3$}\label{sec:h>2}
We divide the proof of the lower bound in Theorem \ref{thm:main1} according to the parity of $h$.  We start with the odd case.

\begin{proposition}\label{prop:lhodd}
Let $h\geq 3$ be a fixed odd number. 
Let $G=G(n, r(n), h)$ be a layered random geometric graph. 	If
$$
r(n)\ge b\, \left(\frac{\log n}{n^{h-1}}\right)^{\frac{1}{2h}},
$$
for $b >\left(\frac{2^{2+3(h-1)}}{\pi^3}\right)^{\frac{1}{2h}}$, 
then $G$ is \whp rainbow connected.	
\end{proposition}

\begin{proof}  Write $h=2w+1$.  
As before, we analyze the behavior in the extreme case. Let us  write  \[r =b\, \left(\frac{\log n}{n^{h-1}}\right)^{\frac{1}{2h}},\]
for some constant $b$ to be determined later.

Denote as $G_i=G_i(n,r)$,  the $i$-th. layer of $G$, $1\le i\le h$.  For a subset $J\subseteq [h]$, let $G_J(n,r)$ be the multilayered graph  formed by the layers included in $J$.   

\remove{For a pair  $v_i,v_j$ of distinct   vertices in $V$ and a permutation $\sigma$ of $\{1,2,3,\dots, h\}$,  let  $P(i,j,\sigma)$ denote the set of  rainbow paths of length $h$ joining $v_i$ and $v_j$, where the $i$-th. edge is in $G_{\sigma(i)}$. Define $J_1(\sigma)=\{\sigma(1),\dots, \sigma(w)\}$, $J_2(\sigma)=\{\sigma(w+2),\ldots ,\sigma (h)\}$.  In the following, we analyze the probability that the edges in $G_{w+1}$ join  paths from $v_i$ following $\sigma$ in $G_{J_1(\sigma)}$  with those ending at $v_j$ in $G_{J_2(\sigma)}$. }

For a permutation $\sigma$ of $\{1,2,3,\dots, h\}$, let $\sigma^R$ be the reversed permutation, i.e. $\sigma^R(i)= \sigma(n-i+1)$.  Let  $P(i,j,\sigma)$ denote the set of  rainbow paths of length $h$ joining $v_i$ and $v_j$, where the $i$-th. edge is in $G_{\sigma(i)}$. Define $J_1(\sigma)=\{\sigma(1),\dots, \sigma(w)\}$, $J_2(\sigma)=\{\sigma(w+2),\ldots ,\sigma (h)\}$.  In the following, we analyze the probability that the edges in $G_{w+1}$ join  rainbow paths of length $w$ starting at $v_i$ and following $\sigma$, thus using colors in $J_1(\sigma)$,  with  rainbow paths of length $w$ starting at $v_j$ and following $\sigma^R$, thus using colors in $J_2(\sigma)$, in order to form a rainbow path in  $P(i,j,\sigma)$.

Let $V_1,V_2$ be a partition of $V$ such that for a pair $v_i,v_j$ of distinct vertices in $V$, $v_i\in V_1$ and $v_j\in V_2$. 
Let $A$ be the set of vertices reachable from $v_i$ by $\sigma$-rainbow paths of length $w$, 
formed by vertices in $V_1$.  
Let $B$ be the set of vertices reached from $v_j$ by $\sigma^R$-rainbow paths of length $w$ formed by vertices in $V_2$.

Observe that $A\cap B = \emptyset$. Furthermore, both $A$ and $B$ can be seen as the sets $\cN_{w,\sigma}(v_i)$ and $\cN_{w,\sigma^R}(v_i)$ respectively, as defined in Equation~(\ref{eq:Nell}), for a multilayered random geometric graph on $m=n/2$ vertices and radius $r$.  Expressing $r$ as a function of $m$ we have that 
\begin{equation*}
r(m) = b\, \left(\frac{\log (2m)}{(2m)^{h-1}}\right)^{\frac{1}{2h}} =  \frac{b}{2^{h-1}}\, \left(\frac{\log (2m)}{(m)^{h-1}}\right)^{\frac{1}{2h}} \sim  \frac{b}{2^{h-1}}\, \left(\frac{\log m }{m^{h-1}}\right)^{\frac{1}{2h}}.
\end{equation*}
Thus the conditions of Proposition~\ref{prop:exp} are meet and we have that \whp 
\begin{align*}
& \frac{\pi}{ 4^w 2} (mr(m)^2)^{w} \leq |A| \leq \left(\frac{3\pi}{2}\right)^{w} (m r(m)^2)^{w} \text{ and } \\
&\frac{\pi}{ 4^w 2} (mr(m)^2)^{w} \leq |B| \leq \left(\frac{3\pi}{2}\right)^{w} (m r(m)^2)^{w}.
\end{align*} 
Rewriting as a function of $n$, 
\begin{equation}\label{eq:AB-LB}
\begin{split}
& \frac{\pi}{ 4^w 2^w 2} (nr^2)^{w} \leq |A| \leq \left(\frac{3\pi}{4}\right)^{w} (n r^2)^{w} \text{ and } \\
&\frac{\pi}{ 4^w 2^w 2} (nr^2)^{w} \leq |B| \leq \left(\frac{3\pi}{4}\right)^{w} (n r^2)^{w}.
\end{split}
\end{equation} 
The remaining analysis is done  by conditioning on this event. So, as $h=2w +1$, we have 
$$\frac{\pi^2}{8^{h-1}2} (nr^2)^{h-1}  \leq |A|\cdot|B| \leq \left(\frac{3\pi}{4}\right)^{h-1} (n r^2)^{h-1}.$$

For a pair $(v_\ell,v_{\ell'})\in A\times B$, let $Y_{\ell,\ell'}$ be the indicator that $v_\ell$ and $v_{\ell'}$ are neighbours in $G_{\sigma (w+1)}$. According to Equation~\eqref{eq:basicprob}, 
$$  \pi r^2 - o(r^2) \leq \Pr(Y_{\ell,\ell'}= 1) \leq \pi r^2.$$

Let $X_{i,j}=|P(i,j,\sigma)|$.  Then,
$$
X_{ij}\leq\sum_{\ell,\ell'\in A\times B} Y_{\ell,\ell'}.
$$
 We observe that the random variables $Y_{\ell,\ell'}$ are independent: when the pairs $(\ell,\ell'), (\lambda,\lambda')$ have no common value it is clear that $Y_{\ell,\ell'}, Y_{\lambda,\lambda'}$ are independent, while if $\ell=\lambda$, say, then $\Pr (Y_{\ell,\ell'}=1, Y_{\ell,\lambda'}=1)$ is the probability that $v_{\ell'}$ and $v_{\lambda'}$ are both adjacent to $v_\ell$, which is the product $\Pr (Y_{\ell,\ell'}=1)\Pr (Y_{\ell,\lambda'}=1)$.

Then, for a pair of vertices $v_i,v_j$ 
not connected by a rainbow path of length $h-1$,  the sets  $A$ and $B$ are disjoint. In this case,  
\begin{align*}
\Pr (X_{i,j}=0)&\leq\Pr (\cap_{\ell,\ell'} \{Y_{\ell,\ell'\in A\times B}=0\})
=\prod_{\ell,\ell'} \Pr (Y_{\ell,\ell'}=0)\\
&\leq\Pr (Y_{\ell,\ell'}=0)^{|A|\cdot|B|}\\
&\le (1-\pi r^2 + o(r^2))^{\frac{\pi^2}{8^{h-1}2} (n r^2)^{h-1}}\\
&\le e^{-\pi \frac{\pi^2}{8^{h-1}2} r^2 (n r^2)^{h-1} + o\left(r^2 (n r^2)^{h-1}\right)}\\
&\leq e^{-\frac{\pi^3}{8^{h-1}2}  n^{h-1}r^{2h}}\, .
\end{align*}
By using the union bound on all pairs $i,j$,
$$
\Pr (\cap_{i,j}\{X_{ij}\ge 1\})=1-\Pr (\cup_{i,j} X_{i,j}=0)\ge 1-n^2e^{\frac{\pi^3}{8^{h-1}2}  n^{h-1}r^{2h}}.
$$
Substituting  $r$ by its value, 
\[n^{h-1}r^{2h}=  b^{2h} \ln n.\]
Therefore, choosing $b$  so that $2-\frac{\pi^3}{8^{h-1}2}  b^{2h} <0$, i.e., 
$$b> \left(\frac{4}{\pi^3} 8^{h-1}\right)^{\frac{1}{2h}} $$  
the last term in the bound on $\Pr (\cap_{i,j}\{X_{ij}\ge 1\})$ 
is $o(1)$. Hence \whp all pairs $v_i,v_j$ are connected by  a rainbow path of length $h$.
\end{proof}

We next consider the case $h$ even.
\begin{proposition}\label{prop:lheven} 
Let $h\geq 3$ be a fixed even number. 
Let $G=G(n, r(n), h)$ be a layered random geometric graph. 	If
$$
r(n)\ge b\, \left(\frac{\log n}{n^{h-1}}\right)^{\frac{1}{2h}},
$$
for $b> \left(\frac{2^{2+3(h-1)}}{\pi^3}\right)^{\frac{1}{2h}}$, 
then $G$ is \whp rainbow connected.	
\end{proposition}

\begin{proof}
  Let $h=2 w$ and let $b$ a constant to be defined later.  The proof is similar to the one for the case $h$ odd.  We analyze the probability that the edges in layer $w+1$ join rainbow paths of length $w$ starting at $v_i$ with rainbow paths of length $w-1$ ending in $v_j$ for a prefixed permutation of color.  The main  difference is in the definition of the sets $A$ and $B$. 

Let $V_1,V_2$ be a balanced partition of $V$ so that $v_i\in v_1$ and $v_j\in V_2$.  
Let $A$  be the set of vertices reached from $v_i$ by rainbow paths of length $w$ starting at $v_i$ following the color order determined by $\sigma$ on vertices from $V_1$ .  Let $B$  be the set of vertices reached from $v_j$ by rainbow paths of length $w-1$ starting at $v_j$  following the reverse order in $\sigma$, the $i$-th. edge colored $\sigma (k-i)$ on vertices from $V_2$.  

In a similar way as done to get 
Equation~(\ref{eq:AB-LB}),
 the multilayer random graphs on $m=n/2$ vertices and radius $r$ verify the conditions of Proposition~\ref{prop:exp}. Then,   \whp 
\begin{align*}
& \frac{\pi}{ 4^w 2} (mr(m)^2)^{w} \leq |A| \leq \left(\frac{3\pi}{2}\right)^{w} (m r(m)^2)^{w} \text{ and } \\
&\frac{\pi}{ 4^{w-1} 2} (mr(m)^2)^{w-1} \leq |B| \leq \left(\frac{3\pi}{2}\right)^{w-1} (m r(m)^2)^{w-1}.
\end{align*} 
Rewriting as a function of $n$, 
\begin{align*}
& \frac{\pi}{ 4^{w} 2^w 2} (nr^2)^{w} \leq |A| \leq \left(\frac{3\pi}{4}\right)^{w} (n r^2)^{w} \text{ and } \\
&\frac{\pi}{ 4^{w-1} 2^{w-1} 2} (nr^2)^{w-1} \leq |B| \leq \left(\frac{3\pi}{4}\right)^{w-1} (n r^2)^{w-1}.
\end{align*} 
Therefore, 
$$\frac{\pi}{8^{2w-1} 2} (nr^2)^{2w-1}  \leq |A|\cdot|B| \leq \left(\frac{3\pi}{2}\right)^{2w-1} (n r^2)^{2w-1}.$$

As, $h=2w$, we get the same lower and upper bound for $|A|\cdot|B|$ as in the  case of $h$ odd.

The remaining part of the proof is conditioned on this fact and is exactly equal to the one for $h$ odd. 

We conclude that, choosing $b$  so that  
$$b> \left(\frac{4}{\pi^3} 8^{h-1}\right)^{\frac{1}{2h}} $$  \whp all pairs $v_i,v_j$ are connected by  a rainbow path of length $h$.
\end{proof}

\section{Upper bound for $h\geq 3$}\label{sec:upper}

For the lower bound  in Theorem~\ref{thm:main1}, we use an induction argument.
\begin{proposition}\label{prop:lboundh} Let $h\ge 3$ be fixed.  Let $G=G(n,r(n),h)$ be a multilayered random geometric graph.
If  
$$
r(n)\le c\left(\frac{\log n}{n^{h-1}}\right)^{1/2h},
$$ 
for   $c < \frac{2}{3\pi} \left(\frac{1}{2\pi}\right)^{(h-1)}$,
then \whp, $G$ is not rainbow connected.
\end{proposition}

\begin{proof} Denote by $G_i=G_i(n,r)$ the $i$-th layer of $G$, $1\le i\le h$, where 
$r=r(n) = c\left(\frac{\log n}{n^{h-1}}\right)^{1/2h}$. The proof is by induction on $h$. By Proposition \ref{prop:main1h=2} the statement holds for $h= 2$ as $\frac{1}{2\pi^2}<0.56$. 

Fix a permutation $\sigma$ of $\{1,2,\ldots ,h\}$   and a vertex $u\in V$. Let $A=\cN_{h-1,\sigma}(u)$ and  $B=V\setminus \cup_{t=0}^{h-1} \cN_{t,\sigma}(u)$.  

According to Proposition~\ref{prop:exp}, \whp, for $1\leq \ell \leq h-1$,  
$$\frac{\pi}{ 4^t 2} (nr^2)^{t} \leq |\cN_{t,\sigma}(u)| \leq \left(\frac{3\pi}{2}\right)^{t} (n r^2)^{t}.$$ 

Conditioning  on this event, we have 
\begin{align*}
|B| &=n-\sum_{t=0}^{h-1} |\cN_{t,\sigma}(u)|\geq n - \sum_{t=0}^{h-1} \left(\frac{3\pi}{2}^t\right)(n r^2)^{t} \\
&\geq n - \left(\frac{3\pi}{2}^{h-1}\right)(n r^2)^{h-1}= n - o(n).
\end{align*} 
Thus, for $n$ sufficiently large, $|B|\ge n/2$.

 For each $v_j\in B$, let $Y_{ij}$ be the indicator function that the neighborhood of $v_j$ in $G_{\sigma (h)}$ is disjoint from $A$. We have $\E(Y_{ij})\ge (1-\pi r^2)^{|A|}$. Consider the random variable  $Y_i=\sum_{v_j\in B}Y_{ij}$ which counts  the number of vertices in $B$ whose neighborhood in $G_{\sigma (h)}$ is disjoint from $A$. By using $|B|\ge n/2$, $1-x \geq e^{-2x}$, the lower bound on $|A|$  and the upper bound on $r$, we have
\begin{align*}
\E(Y_i)&=\sum_{v_j\in B}\E(Y_{ij})
\ge |B|(1-\pi r^2) ^{|A|}\\
&\ge \frac{n}{2}e^{-2\pi r^2|A|}
\ge \frac{n}{2}e^{-\left(2 \pi r^2\left(\frac{3\pi}{2}\right)^{h-1} (nr^2)^{h-1}\right)}\\
&\ge \frac{1}{2}e^{\left(1- 2\pi \left(\frac{3\pi}{2}\right)^{h-1} c^{h-1}\right) \ln n}\, .
\end{align*}
By choosing $c$ such that $1- 2\pi \left(\frac{3\pi}{2}\right)^{h-1} c^{h-1} > 0$, i.e., 
$$c < \frac{2}{3\pi} \left(\frac{1}{2\pi}\right)^{(h-1)},$$
we have $\E(Y_i)\to \infty $ as $n\to \infty$. 

Since the variables $Y_{ij}$ are independent for $v_j\in B$, by Chernoff's bound, we have $\Pr (Y_i=0)\le e^{-\E(Y_i)/2}=o(1)$. Therefore, \whp there are pairs $v_i,v_j$ not joined by rainbow paths with the $i$-th edge along the path from $v_i$ to $v_j$ in $G_{\sigma (i)}$. By repeating the same argument for all $h!$ permutations $\sigma$, \whp the graph $G$ is not rainbow connected. Finally, by monotonicity of the property, we get the result.
\end{proof}

\section{Conclusions}


The main purpose of this paper is to identify the threshold for radius $r(n)$ to get a rainbow-connected multilayered random geometric graph, as obtained in Theorem \ref{thm:main1}. As mentioned in Section~\ref{sec:intro}, the analogous problem of determining the threshold for $h$ in order for the multilayered binomial random graph is rainbow connected  was addressed by Bradshaw and Mohar~\cite{bradshaw2021}. 

We find the model of multilayered random geometric graphs highly appealing, because  
it leads to a host of interesting problems. One may think of a dynamic setting, where $n$ individuals perform random walks within the cube and communicate with  close neighbors at discrete times $t_1<t_2<\cdots <t_h$. The rainbow connectivity in this setting measures the number of instants needed so that every individual can communicate with each  other one. A natural immediate extension is to find  the threshold for getting  rainbow connectivity $k$, as achieved in the case of multilayered binomial random graphs by Shang~\cite{Shang2023}. 

There is a vast literature addressing rainbow problems in random graph models, and this paper is meant to open the path to these problems in the context of multilayered random geometric graphs. 

 It would also be interesting to find asymptotic estimates on $r$ such that $h$ copies produce a rainbow clique of size $\sqrt{h}$.

Observe that for large $h$, the threshold of $r$ for rainbow connectivity approaches the connectivity threshold of random geometric graphs. However, the arguments in the proof apply only for constant $h$. For $h$ growing as a function of $n$, the correlation between distinct edges in our model decrease, and the model gets closer to the random binomial graph, where the results are expected to behave differently and the geometric aspects of the model become irrelevant.

\section*{Acknowledgments} 
We want to thank the anonymous referee for the careful reading and helpful suggestions.

\bibliographystyle{abbrv}
\bibliography{biblio.bib}

\end{document}